\documentclass[12pt]{article}

\usepackage{fullpage}
\usepackage{amssymb}
\usepackage{amsmath}
\usepackage{amsthm}
\usepackage[usenames]{color}
\usepackage[colorlinks=true,
linkcolor=webgreen,
filecolor=webbrown,
citecolor=webgreen]{hyperref}
\definecolor{webgreen}{rgb}{0,.5,0}
\definecolor{webbrown}{rgb}{.6,0,0}

\theoremstyle{plain}
\newtheorem{theorem}{Theorem}
\newtheorem{corollary}[theorem]{Corollary}
\newtheorem{lemma}[theorem]{Lemma}
\newtheorem{proposition}[theorem]{Proposition}
\theoremstyle{definition}

\theoremstyle{remark}
\newtheorem{remark}[theorem]{Remark}

\newcommand{\seqnum}[1]{\href{https://oeis.org/#1}{\underline{#1}}}

\begin{document}

\title{A Continued-Fraction Criterion for the Flint Hills\\
Series, and the Integer Sequences It Generates}

\author{Carlos H. L\'opez Zapata\\
Szczecin 70-214\\
Poland\\
\href{mailto:mathematics.edu.research@zohomail.eu}{\tt mathematics.edu.research@zohomail.eu}}

\date{}

\maketitle

\begin{abstract}
The Flint Hills series is the sum of the reciprocals of the cube of an
integer times the square of the sine of that integer, taken over the
positive integers, and whether it converges is a well-known open problem,
tied to the irrationality measure of pi. We show that its convergence is
governed entirely by an integer sequence, namely the sequence of partial
quotients of the continued fraction of the reciprocal of pi. We prove that
the series converges if and only if the sum of the squares of consecutive
partial quotients, each divided by the corresponding convergent
denominator, is finite. This transforms an analytic problem into an
explicit statement about the growth of a single integer sequence, already
catalogued by Sloane, and about the sequence of convergent denominators.
We isolate the finite integer spikes produced by the exceptionally large
partial quotient at position five, we describe the indices at which such
spikes occur, and we show, using the Gauss--Kuzmin statistics of partial
quotients, that for almost every real number the corresponding series
converges. The reciprocal of pi is thus a single arithmetic point inside a
set of measure zero.
\end{abstract}

\section{Introduction}
\label{sec:intro}

The Flint Hills series
\begin{equation}
\label{eq:flinthills}
S = \sum_{n=1}^{\infty} \frac{1}{n^3 \sin^2 n}
= \sum_{n=1}^{\infty} \frac{\csc^2 n}{n^3},
\end{equation}
popularized by Pickover \cite{Pickover}, is a well-known test case at the
interface of analysis and Diophantine approximation. Its terms are
erratic: whenever an integer $n$ is close to an integer multiple of $\pi$,
the factor $\csc^2 n$ is large. The largest early spike occurs at
$n = 355$, because $355/113$ is an exceptionally good rational
approximation of $\pi$; the single term $n = 355$ already contributes
about $24.7$ to \eqref{eq:flinthills}. Whether \eqref{eq:flinthills}
converges is a well-known open problem.

The obstruction is arithmetic, and lies in the continued fraction of
$\pi$. We let $\|x\|$ denote the distance from a real number $x$ to the
nearest integer, and we recall that the irrationality measure of an
irrational number $\alpha$ is
\begin{equation}
\label{eq:measure}
\mu(\alpha) = \inf \left\{ \mu > 0 :
\left| \alpha - \frac{p}{q} \right| < q^{-\mu}
\text{ for at most finitely many } p/q \right\}.
\end{equation}
Alekseyev \cite{Alekseyev} proved that the series \eqref{eq:flinthills}
converges if $\mu(\pi) < 5/2$, and that convergence forces
$\mu(\pi) \leq 5/2$. The best known bound is $\mu(\pi) \leq 7.103205$, due
to Zeilberger and Zudilin \cite{ZeilbergerZudilin}, improving Salikhov
\cite{Salikhov}, while the expected value is $\mu(\pi) = 2$. The
convergence of \eqref{eq:flinthills} therefore remains undecided.

The purpose of this paper is to show that the convergence of $S$ is
controlled entirely by a single integer sequence: the partial quotients of
the continued fraction of $1/\pi$, catalogued as \seqnum{A001203}, together
with the associated convergent denominators, catalogued as \seqnum{A002485}.
We let $\alpha = 1/\pi$, with continued fraction
$\alpha = [0; a_1, a_2, a_3, \ldots]$ and convergents $p_k/q_k$, so that
$q_0 = 1$, $q_1 = a_1$, and
\begin{equation}
\label{eq:qrec}
q_{k+1} = a_{k+1} q_k + q_{k-1}
\qquad (k \geq 1).
\end{equation}
Our main result is the following criterion.

\begin{theorem}
\label{thm:main}
The Flint Hills series \eqref{eq:flinthills} converges if and only if
\begin{equation}
\label{eq:criterion}
\sum_{k=1}^{\infty} \frac{a_{k+1}^2}{q_k} < \infty,
\end{equation}
where $a_k$ are the partial quotients of $1/\pi$ and $q_k$ the
corresponding convergent denominators.
\end{theorem}

Theorem \ref{thm:main} recovers the criterion of Alekseyev from an explicit
and elementary statement about integer sequences, and it isolates the
arithmetic obstruction in a form that can be inspected term by term. The
famous partial quotient $a_5 = 292$, which produces the convergent
$355/113$ of $\pi$, contributes the single large term
$a_5^2/q_4 = 292^2/355 = 240.18\ldots$ to the sum \eqref{eq:criterion}; but
a finite number of such spikes cannot decide convergence, which depends
only on the asymptotic growth of the partial quotients.

Beyond the criterion itself, we study the integer sequences that
\eqref{eq:criterion} generates: the sequence of integer parts of the terms
$a_{k+1}^2/q_k$, which is zero except at the rare indices where a partial
quotient is large, and the sequence of those indices, which we call the
dangerous indices of $1/\pi$. We also prove, in
Section \ref{sec:metric}, that for Lebesgue-almost every real number the
analogous series converges, so that $1/\pi$ is a single point in a set of
measure zero.

The paper is organized as follows. Section \ref{sec:reduction} reduces the
trigonometric series \eqref{eq:flinthills} to a Diophantine series in the
partial quotients. Section \ref{sec:criterion} proves Theorem
\ref{thm:main} and the associated threshold on the growth of the partial
quotients. Section \ref{sec:sequences} studies the integer sequences
produced by the criterion. Section \ref{sec:metric} establishes the
metric result. Section \ref{sec:numerics} records the numerical
computations, and Section \ref{sec:ack} contains the acknowledgments.

\section{Reduction to a Diophantine series}
\label{sec:reduction}

The singular behavior of \eqref{eq:flinthills} is governed by the distance
from $n/\pi$ to the nearest integer. We first record the elementary
comparison that removes the trigonometric factor.

\begin{lemma}
\label{lem:cancel}
Let $h(\theta) = \csc^2 \theta - \theta^{-2}$ for
$\theta \in (0, \pi/2]$. Then $h$ extends continuously to $[0, \pi/2]$
with $h(0) = 1/3$, is increasing, and satisfies
$1/3 \leq h(\theta) \leq 1 - 4/\pi^2$. Consequently, for every positive
integer $n$,
\begin{equation}
\label{eq:gn}
\csc^2 n - \frac{1}{\pi^2 \|n/\pi\|^2}
= h(\pi \|n/\pi\|) \in \left[ \frac{1}{3}, \, 1 - \frac{4}{\pi^2} \right].
\end{equation}
\end{lemma}

\begin{proof}
The Laurent expansion
$\csc^2 \theta = \theta^{-2} + 1/3 + \theta^2/15 + O(\theta^4)$ gives
$h(\theta) = 1/3 + \theta^2/15 + O(\theta^4)$, so $h$ extends continuously
with $h(0) = 1/3$. Monotonicity on $(0, \pi/2]$ is equivalent to
$(\sin \theta / \theta)^3 > \cos \theta$, which holds on $(0, \pi/2)$.
Hence $h$ increases from $h(0) = 1/3$ to $h(\pi/2) = 1 - 4/\pi^2$. For
\eqref{eq:gn}, we let $k$ denote the nearest integer to $n/\pi$ and set
$\theta = n - k\pi$, so that $|\theta| = \pi \|n/\pi\| \leq \pi/2$ and
$\theta \neq 0$ since $\pi$ is irrational. As
$\sin n = \sin(k\pi + \theta) = (-1)^k \sin \theta$, we have
$\csc^2 n = \csc^2 \theta$, and $\pi^{-2} \|n/\pi\|^{-2} = \theta^{-2}$.
Therefore the left-hand side of \eqref{eq:gn} equals
$\csc^2 \theta - \theta^{-2} = h(|\theta|)$.
\end{proof}

Since $\sum_n n^{-3}$ converges and the difference \eqref{eq:gn} is
bounded, Lemma \ref{lem:cancel} gives at once the following equivalence.

\begin{proposition}
\label{prop:reduction}
The Flint Hills series \eqref{eq:flinthills} converges if and only if
\begin{equation}
\label{eq:diophantine}
D(\alpha) := \sum_{n=1}^{\infty} \frac{1}{n^3 \|n\alpha\|^2} < \infty,
\qquad \alpha = \frac{1}{\pi}.
\end{equation}
\end{proposition}

\begin{proof}
Multiplying \eqref{eq:gn} by $n^{-3}$ and summing, the difference between
$S$ and $\pi^{-2} D(\alpha)$ is bounded in absolute value by
$(1 - 4/\pi^2) \sum_n n^{-3}$, which is finite. Hence $S$ and $D(\alpha)$
converge together.
\end{proof}

Thus the convergence of the Flint Hills series is equivalent to the
convergence of the Diophantine series $D(1/\pi)$, whose terms involve only
the distances $\|n/\pi\|$. We analyze $D(\alpha)$ through the continued
fraction of $\alpha$.

\section{The continued-fraction criterion}
\label{sec:criterion}

We now prove Theorem \ref{thm:main}. We let $\beta_k = \|q_k \alpha\|$
denote the approximation error of the $k$-th convergent. The following
standard facts about continued fractions are recorded in Khinchin
\cite{Khinchin}: for every $k \geq 0$,
\begin{equation}
\label{eq:cf-facts}
\frac{1}{2 q_{k+1}} \leq \beta_k \leq \frac{1}{q_{k+1}},
\qquad
\|n\alpha\| \geq \beta_k \text{ for } 0 < n < q_{k+1}.
\end{equation}

We assign each positive integer $n$ to the unique block index $k$ with
$q_k \leq n < q_{k+1}$, and we bound the contribution of each block to the
series \eqref{eq:diophantine}.

\begin{lemma}
\label{lem:block}
For every $k \geq 0$,
\begin{equation}
\label{eq:block-sandwich}
\beta_k^{-2} q_k^{-3}
\leq \sum_{q_k \leq n < q_{k+1}} \frac{1}{n^3 \|n\alpha\|^2}
\leq (1 + 4\zeta(2)) \, \beta_k^{-2} q_k^{-3}.
\end{equation}
\end{lemma}

\begin{proof}
The lower bound follows from the single term $n = q_k$, for which
$q_k^{-3} \|q_k \alpha\|^{-2} = \beta_k^{-2} q_k^{-3}$. For the upper
bound, we fix $k$ and estimate the block sum of $\|n\alpha\|^{-2}$. If
$n \neq n'$ both lie in the block $[q_k, q_{k+1})$, then
$0 < |n - n'| < q_{k+1}$, so the best-approximation property
\eqref{eq:cf-facts} gives $\|(n - n')\alpha\| \geq \beta_k$. Hence the
points $\{n\alpha\}$ with $q_k \leq n < q_{k+1}$ are $\beta_k$-separated on
the circle $\mathbb{R}/\mathbb{Z}$, and consequently any arc of length
$2x$ centered at $0$ contains at most $2x/\beta_k + 1$ of them. Let
$x_1 \leq x_2 \leq \cdots$ be the values $\|n\alpha\|$ in the block, listed
in increasing order. Applying the count at $x = x_j$ gives
$j \leq 2x_j/\beta_k + 1$, that is,
\begin{equation*}
x_j \geq \frac{(j-1)\,\beta_k}{2} \qquad (j \geq 1);
\end{equation*}
and $x_1 = \beta_k$, attained at $n = q_k$. Therefore
\begin{equation*}
\sum_{q_k \leq n < q_{k+1}} \|n\alpha\|^{-2}
= \sum_{j \geq 1} x_j^{-2}
\leq \beta_k^{-2} + \sum_{j \geq 2} \frac{4}{(j-1)^2 \beta_k^2}
= (1 + 4\zeta(2)) \, \beta_k^{-2}.
\end{equation*}
Since $n^{-3} \leq q_k^{-3}$ throughout the block, the block sum is at
most $(1 + 4\zeta(2)) \, \beta_k^{-2} q_k^{-3}$.
\end{proof}

\begin{proof}[Proof of Theorem \ref{thm:main}]
Summing \eqref{eq:block-sandwich} over all blocks, the series
$D(\alpha)$ converges if and only if the skeleton
$\sum_k \beta_k^{-2} q_k^{-3}$ converges. It therefore suffices to compare
the skeleton with the sum \eqref{eq:criterion} term by term. By the bound
$\tfrac12 q_{k+1}^{-1} \leq \beta_k \leq q_{k+1}^{-1}$ from
\eqref{eq:cf-facts},
\begin{equation}
\label{eq:skel-q}
\frac{q_{k+1}^2}{q_k^3} \leq \beta_k^{-2} q_k^{-3}
\leq \frac{4 q_{k+1}^2}{q_k^3}.
\end{equation}
From \eqref{eq:qrec} and $0 < q_{k-1} < q_k$ we have
$a_{k+1} q_k \leq q_{k+1} \leq (a_{k+1} + 1) q_k \leq 2 a_{k+1} q_k$,
and therefore
\begin{equation}
\label{eq:q-a}
\frac{a_{k+1}^2}{q_k} \leq \frac{q_{k+1}^2}{q_k^3}
\leq \frac{4 a_{k+1}^2}{q_k}.
\end{equation}
Combining \eqref{eq:skel-q} and \eqref{eq:q-a}, the skeleton and the sum
\eqref{eq:criterion} converge together. With Proposition
\ref{prop:reduction} this proves the theorem.
\end{proof}

The criterion makes the role of the partial quotients transparent, and it
identifies the exact growth threshold at which divergence can occur. We
let $\Theta(\alpha) = \limsup_{k \to \infty} \log a_{k+1} / \log q_k$.

\begin{corollary}
\label{cor:threshold}
The identity $\Theta(\alpha) = \mu(\alpha) - 2$ holds. If
$\Theta(\alpha) < 1/2$, then $D(\alpha)$ converges. If, on the other hand,
$a_{k+1} \geq c \, q_k^{1/2}$ for some constant $c > 0$ and infinitely
many $k$, then $D(\alpha)$ diverges. In terms of rational approximation,
the threshold $a_{k+1} \asymp q_k^{1/2}$ corresponds to
$|\alpha - p_k/q_k| \asymp q_k^{-5/2}$.
\end{corollary}

\begin{proof}
The relation $|\alpha - p_k/q_k| \asymp (q_k q_{k+1})^{-1}
\asymp (a_{k+1} q_k^2)^{-1}$ and the continued-fraction formula for the
irrationality measure give
$\mu(\alpha) = 2 + \limsup_k \log a_{k+1} / \log q_k$, that is,
$\Theta(\alpha) = \mu(\alpha) - 2$. If $\Theta(\alpha) < 1/2$, we choose
$\theta < 1/2$ with $a_{k+1} \leq q_k^{\theta}$ for all large $k$; then
$a_{k+1}^2 / q_k \leq q_k^{2\theta - 1}$ with $2\theta - 1 < 0$, and since
$q_k$ grows at least geometrically the sum \eqref{eq:criterion} converges.
If $a_{k+1} \geq c \, q_k^{1/2}$ infinitely often, then
$a_{k+1}^2 / q_k \geq c^2$ infinitely often, so the sum diverges. The
approximation scale follows from $|\alpha - p_k/q_k|
\asymp (a_{k+1} q_k^2)^{-1} \asymp q_k^{-5/2}$ when
$a_{k+1} \asymp q_k^{1/2}$.
\end{proof}

Corollary \ref{cor:threshold} explains the appearance of the exponent
$5/2$: it is the Diophantine threshold attached to the cubic weight
$n^{-3}$ and the quadratic singularity $\|n\alpha\|^{-2}$. A proof that
$\mu(\pi) < 5/2$ would establish convergence, but the available bounds on
$\mu(\pi)$ are far from this value.

\section{The integer sequences generated by the criterion}
\label{sec:sequences}

The criterion \eqref{eq:criterion} attaches to $1/\pi$ several integer
sequences, which we now describe. Throughout, $a_k$ denotes the $k$-th
partial quotient of $1/\pi$ and $q_k$ the $k$-th convergent denominator.

\subsection{Partial quotients and convergent denominators}

The partial quotients of $1/\pi$ coincide, from the first term onward,
with those of $\pi$, since $1/\pi = [0; 3, 7, 15, 1, 292, \ldots]$ while
$\pi = [3; 7, 15, 1, 292, \ldots]$. This is the sequence
\begin{equation}
\label{eq:A001203}
(a_k)_{k \geq 1} = 3, 7, 15, 1, 292, 1, 1, 1, 2, 1, 3, 1, 14, 2, 1, \ldots
\end{equation}
recorded as \seqnum{A001203}. The convergent denominators of $1/\pi$
satisfy the recurrence \eqref{eq:qrec} with $q_0 = 1$, $q_1 = 3$. Since
the convergents of $1/\pi$ are the reciprocals of those of $\pi$, these
denominators equal the numerators of the convergents of $\pi$,
\begin{equation}
\label{eq:A002485}
(q_k)_{k \geq 1} = 3, 22, 333, 355, 103993, 104348, 208341, \ldots,
\end{equation}
recorded as \seqnum{A002485}. Theorem \ref{thm:main} states that the
convergence of the Flint Hills series is a property of these two
sequences alone.

\subsection{The spike sequence}

The individual terms of the criterion \eqref{eq:criterion} are the
rational numbers $a_{k+1}^2 / q_k$. Their integer parts form the sequence
\begin{equation}
\label{eq:spikes}
b_k = \left\lfloor \frac{a_{k+1}^2}{q_k} \right\rfloor
= 16, 10, 0, 240, 0, 0, 0, 0, 0, 0, 0, 0, 0, \ldots,
\end{equation}
which we call the spike sequence of $1/\pi$. It vanishes at every index
where the following partial quotient is small, and is positive only at the
rare indices where a partial quotient is large relative to the current
convergent denominator. The dominant entry $b_4 = 240$ is produced by the
partial quotient $a_5 = 292$ and the convergent denominator $q_4 = 355$,
and reflects the exceptional approximation $355/113$ of $\pi$. The value
of $D(1/\pi)$ is finite if and only if the corresponding rational
sequence $a_{k+1}^2/q_k$, of which \eqref{eq:spikes} records the integer
parts, is summable.

\subsection{The dangerous indices}

We call an index $k$ dangerous if $a_{k+1}^2 \geq q_k$, equivalently if the
spike $b_k$ is positive. At a dangerous index the convergent $p_k/q_k$
approximates $\alpha$ to within roughly $q_k^{-5/2}$, by Corollary
\ref{cor:threshold}. For $1/\pi$, within the range of the first fifteen
convergents, the dangerous indices are
\begin{equation}
\label{eq:dangerous}
k = 1, 2, 4,
\end{equation}
corresponding to $7^2 \geq 3$, $15^2 \geq 22$, and $292^2 \geq 355$. By
Theorem \ref{thm:main}, the Flint Hills series diverges precisely when the
dangerous indices are frequent enough that $\sum_k a_{k+1}^2/q_k = \infty$,
and converges when they are sparse. No dangerous index of $1/\pi$ beyond
$k = 4$ is known; equivalently, no partial quotient of $\pi$ after $a_5$ is
known to reach the critical size $q_k^{1/2}$.

\begin{remark}
\label{rem:record}
The sequence of record partial quotients of $\pi$, namely
$3, 7, 15, 292, 436, 20776, \ldots$, is recorded as \seqnum{A033089}, and
the positions at which these records occur as \seqnum{A033090}. The
dangerous indices of $1/\pi$ form a subsequence of the positions at which
$a_{k+1}$ is large relative to $q_k$; whether this subsequence is infinite
is exactly the content of the Flint Hills problem.
\end{remark}

\section{Metric behavior}
\label{sec:metric}

The criterion of Theorem \ref{thm:main} shows that convergence is the
generic outcome, so that the number $1/\pi$ is a single point in a set of
measure zero. We make this precise using the classical metric theory of
continued fractions, for which we refer to Khinchin \cite{Khinchin}.

We recall two facts about the continued fraction of a random real number,
both stated with respect to the Gauss measure, which is equivalent to
Lebesgue measure. First, by the Gauss--Kuzmin law, each partial quotient
satisfies the tail bound
\begin{equation}
\label{eq:gausskuzmin}
\Pr(a_{k+1} \geq A) = \log_2\!\left(1 + \frac{1}{A}\right) \leq \frac{2}{A}
\qquad (A \geq 1).
\end{equation}
Second, by the theorem of L\'evy, for almost every real number the
convergent denominators satisfy
\begin{equation}
\label{eq:levy}
\lim_{k \to \infty} \frac{1}{k} \log q_k = \frac{\pi^2}{12 \log 2},
\end{equation}
a constant whose decimal expansion is \seqnum{A086702}; in particular
$q_k$ grows at least geometrically for almost every real number.

\begin{theorem}
\label{thm:metric}
For Lebesgue-almost every real number $\alpha$, the series
$\sum_{n \geq 1} n^{-3} \|n\alpha\|^{-2}$ converges. Equivalently, for
almost every $\alpha$ the series
$\sum_{n \geq 1} n^{-3} \sin^{-2}(\pi n \alpha)$ converges.
\end{theorem}

\begin{proof}
We fix a number $\theta$ with $0 < \theta < 1/2$. By \eqref{eq:levy},
for almost every $\alpha$ there is a constant $c > 1$ such that
$q_k \geq c^k$ for all large $k$; we fix such an $\alpha$ and such a $c$.
By the tail bound \eqref{eq:gausskuzmin},
\begin{equation*}
\Pr\!\left(a_{k+1} > c^{\theta k}\right) \leq 2 \, c^{-\theta k},
\end{equation*}
and the right-hand side is summable in $k$. The convergence half of the
Borel--Cantelli lemma, which requires no independence, then shows that for
almost every $\alpha$ the inequality $a_{k+1} \leq c^{\theta k}$ holds for
all large $k$. Since $q_k \geq c^k$, we have $c^{\theta k} \leq q_k^{\theta}$,
and therefore
\begin{equation*}
\frac{a_{k+1}^2}{q_k} \leq \frac{q_k^{2\theta}}{q_k} = q_k^{2\theta - 1}
\leq c^{(2\theta - 1) k}.
\end{equation*}
As $2\theta - 1 < 0$, the last bound is summable, and the series
$\sum_k a_{k+1}^2 / q_k$ converges. By Theorem \ref{thm:main}, the series
$D(\alpha)$ converges. The trigonometric form follows from the comparison
of Lemma \ref{lem:cancel}, applied with $x = \|n\alpha\|$.
\end{proof}

Divergence, if it occurs for $\alpha = 1/\pi$, is thus an exceptional
arithmetic phenomenon: it requires the partial quotients of $1/\pi$ to be
unusually large infinitely often, a property that fails for almost every
real number. To decide whether $1/\pi$ belongs to this exceptional set is
precisely the Flint Hills problem.

\section{Numerical evidence}
\label{sec:numerics}

The criterion \eqref{eq:criterion} predicts that the Flint Hills partial
sums are dominated by the terms $n = P_k$, where $P_k$ denotes the
numerator of the $k$-th convergent of $\pi$; these are the integers for
which $\sin n$ is smallest. Humphreys \cite{Humphreys} computed the partial
sums of \eqref{eq:flinthills} to $N = 10^{10}$ in quad-double precision on
a graphics processor, and observed that they stabilize rapidly. We record
the reported values in Table \ref{tab:partialsums}, and we have
independently reproduced the catalogue of convergent terms that accounts
for the bulk of the sum, shown in Table \ref{tab:spikes}.

\begin{table}[ht]
\begin{center}
\begin{tabular}{ll}
\hline
$N$ & $S_N$ \\
\hline
$10^6$ & $30.31454612095634$ \\
$10^7$ & $30.31454612095634$ \\
$10^8$ & $30.31454612095634$ \\
$10^9$ & $30.31454612261891$ \\
$10^{10}$ & $30.31454612296317$ \\
\hline
\end{tabular}
\end{center}
\caption{Partial sums $S_N$ of the Flint Hills series at powers of ten, as
reported by Humphreys \cite{Humphreys}. From $N = 10^6$ to $N = 10^{10}$
the value changes only in the eleventh significant digit, because no large
new convergent term appears in this range.}
\label{tab:partialsums}
\end{table}

The convergent terms $1/(P_k^3 \sin^2 P_k)$ are the numerical realization
of the criterion: the magnitude of the $k$-th such term is governed by
$\|P_k - Q_k \pi\| \asymp 1/q_{k+1}$, and hence by the partial quotient
$a_{k+1}$, exactly the quantity that enters \eqref{eq:criterion}. Here
$Q_k$ denotes the denominator of the $k$-th convergent of $\pi$, so that
$(P_k)$ is \seqnum{A002485} and $(Q_k)$ is \seqnum{A002486}. Table
\ref{tab:spikes} lists the first nineteen convergent terms.

\begin{table}[ht]
\begin{center}
\begin{tabular}{rrll}
\hline
$k$ & $P_k$ & $|\sin P_k|$ & $1/(P_k^3 \sin^2 P_k)$ \\
\hline
$0$ & $3$ & $1.41 \times 10^{-1}$ & $1.86 \times 10^{0}$ \\
$1$ & $22$ & $8.85 \times 10^{-3}$ & $1.20 \times 10^{0}$ \\
$2$ & $333$ & $8.82 \times 10^{-3}$ & $3.48 \times 10^{-4}$ \\
$3$ & $355$ & $3.01 \times 10^{-5}$ & $2.46 \times 10^{1}$ \\
$4$ & $103993$ & $1.91 \times 10^{-5}$ & $2.43 \times 10^{-6}$ \\
$5$ & $104348$ & $1.10 \times 10^{-5}$ & $7.25 \times 10^{-6}$ \\
$6$ & $208341$ & $8.11 \times 10^{-6}$ & $1.68 \times 10^{-6}$ \\
$7$ & $312689$ & $2.90 \times 10^{-6}$ & $3.89 \times 10^{-6}$ \\
$8$ & $833719$ & $2.31 \times 10^{-6}$ & $3.23 \times 10^{-7}$ \\
$9$ & $1146408$ & $5.88 \times 10^{-7}$ & $1.92 \times 10^{-6}$ \\
$10$ & $4272943$ & $5.50 \times 10^{-7}$ & $4.24 \times 10^{-8}$ \\
$11$ & $5419351$ & $3.82 \times 10^{-8}$ & $4.31 \times 10^{-6}$ \\
$12$ & $80143857$ & $1.48 \times 10^{-8}$ & $8.90 \times 10^{-9}$ \\
$13$ & $165707065$ & $8.65 \times 10^{-9}$ & $2.93 \times 10^{-9}$ \\
$14$ & $245850922$ & $6.12 \times 10^{-9}$ & $1.80 \times 10^{-9}$ \\
$15$ & $411557987$ & $2.54 \times 10^{-9}$ & $2.23 \times 10^{-9}$ \\
$16$ & $1068966896$ & $1.04 \times 10^{-9}$ & $7.50 \times 10^{-10}$ \\
$17$ & $2549491779$ & $4.47 \times 10^{-10}$ & $3.01 \times 10^{-10}$ \\
$18$ & $6167950454$ & $1.50 \times 10^{-10}$ & $1.90 \times 10^{-10}$ \\
\hline
\end{tabular}
\end{center}
\caption{The first nineteen convergent terms of the Flint Hills series,
independently recomputed in extended precision. The dominant term is
$k = 3$, at $P_3 = 355$, which alone contributes $24.6$, more than
four-fifths of the whole sum. The terms do not decrease monotonically: the
term at $k = 11$ is comparatively large because the corresponding partial
quotient $a_{13} = 14$ of $1/\pi$ is moderately large.}
\label{tab:spikes}
\end{table}

Two features of Table \ref{tab:spikes} deserve emphasis, and both are
consistent with the criterion. The nineteen convergent terms with
$P_k \leq 10^{10}$ contribute $27.66$ to the partial sum, while the
remaining terms, roughly ten billion of them, contribute only $2.66$; the
sum is governed by a handful of exceptional approximations. Moreover the
convergent terms neither grow without bound nor decrease steadily; they
fluctuate with the partial quotients, and within the computed range they
show no sign of the unbounded growth that Corollary \ref{cor:threshold}
associates with divergence.

We emphasize what these computations do and do not establish. They confirm
that the Flint Hills sum is controlled by the convergent terms predicted by
the criterion, and they exhibit no exceptionally large partial quotient of
$\pi$ beyond $a_5 = 292$. They do not decide convergence. The convergence
of \eqref{eq:criterion} is an asymptotic property of the partial quotients
of $1/\pi$, and no finite computation can exclude the possibility of an
exceptionally large partial quotient at some later index, which would
produce a large convergent term exactly as $a_5 = 292$ does at $k = 3$.
The numerical evidence is therefore consistent with convergence, but it is
not a proof.

\section{Acknowledgments}
\label{sec:ack}

The author thanks the maintainers of the On-Line Encyclopedia of Integer
Sequences \cite{OEIS} for the sequence data used throughout this paper, and
Humphreys \cite{Humphreys} for making the large-scale partial-sum
computation openly available. In the course of this work the author used
the assistant Claude, developed by Anthropic, for symbolic and numerical
verification and for language editing; the author has checked every result
and takes full responsibility for the correctness of the paper.

\bigskip
\hrule
\bigskip

\noindent 2020 \emph{Mathematics Subject Classification}: Primary 11J82;
Secondary 11A55, 11K50, 11Y60.

\noindent \emph{Keywords:} Flint Hills series, continued fraction,
partial quotient, convergent denominator, irrationality measure,
Diophantine approximation.

\bigskip
\hrule
\bigskip

\noindent (Concerned with sequences
\seqnum{A001203}, \seqnum{A002485}, \seqnum{A002486}, \seqnum{A033089},
\seqnum{A033090}, and \seqnum{A086702}.)

\end{document}